\documentclass[a4paper]{amsart}
\usepackage{textcomp, amsmath, amssymb, amsthm}
\usepackage{graphicx}
\usepackage{xcolor}
\usepackage{enumerate}
\usepackage{mathrsfs}
\usepackage{float}
\usepackage[makeroom]{cancel}

\newtheorem*{theorem*}{Theorem}
\newtheorem{theorem}{Theorem}[section]
\newtheorem{lemma}[theorem]{Lemma}
\newtheorem*{lemma*}{Lemma}

\newtheorem{claim}[theorem]{Claim}
\newtheorem{prop}[theorem]{Proposition}

\theoremstyle{definition} \newtheorem{remark}[theorem]{Remark}

\DeclareMathOperator{\proj}{proj}
\DeclareMathOperator{\prx}{pr_1}
\DeclareMathOperator{\pry}{pr_2}
\DeclareMathOperator{\lar}{large}
\DeclareMathOperator{\sma}{small}

\def\rr{{\mathbb R}}

\def\cc{{\mathbb C}}
\def\nn{{\mathbb N}}

\def\su{\subset}

\def\al{\alpha}
\def\be{\beta}
\def\ga{\gamma}
\def\Ga{\Gamma}
\def\de{\delta}

\def\om{\omega}
\def\la{\lambda}
\def\ep{\varepsilon}
\def\si{\sigma}

\def\ti{\widetilde}
\def\diam{{\rm diam}\, }
\def\dim{{\rm dim}\, }

\def\leb{\mathcal{L}}
\def\hau{\mathcal{H}}

\def\lkb{\lesssim}
\def\gkb{\gtrsim}

\def\phi{\varphi}
\def\te{\theta}
\def\rde{\de^s}

\providecommand{\semmi}[1]{}
\begin{document}

\title[Hausdorff dimension of Besicovitch sets of Cantor graphs]{Hausdorff dimension of Besicovitch sets of Cantor graphs}
\author{Iqra Altaf, Marianna Cs\"ornyei, Korn\'elia H\'era}

\address
{Department of Mathematics, The University of Chicago,
5734 S. University Avenue, Chicago, IL 60637, USA}

\email{iqra@uchicago.edu}

\email{csornyei@math.uchicago.edu}

\address
{Mathematical Institute, University of Bonn,
Endenicher Allee 60, 53115 Bonn, Germany
}

\email{herakornelia@gmail.com}

\keywords{Hausdorff dimension, Besicovitch set, rectifiable curve, Cantor staircase}

\subjclass[2010]{28A78}

\thanks{The third author was supported  
by the 
Hungarian National Research, Development and Innovation Office - NKFIH, 124749.} 

\begin{abstract}
We consider the Hausdorff dimension of planar Besicovitch sets for rectifiable sets $\Ga$, i.e. sets that contain a rotated copy of $\Ga$ in each direction. We show that for a large class of Cantor sets $C$ and Cantor-graphs $\Ga$ built on $C$, the Hausdorff dimension of any $\Ga$-Besicovitch set must be at least $\min\left(2-s^2,\frac{1}{s}\right)$, where $s=\dim C$. 
\end{abstract}

\maketitle

\section{Introduction}
\label{introd}

It is well known that Besicovitch sets in the plane - sets containing a unit line segment in every direction - can have zero Lebesgue measure, and must have Hausdorff dimension 2. It is natural to ask what happens for other curves in place of the line segment. For circular arcs, the problem is trivial, as we can rotate the circular arc around the center of the circle to obtain a Besicovitch set for that circular arc, of dimension 1. For other curves $\Ga$, it is worthwhile to ask first whether it is possible to obtain $\Ga$-Besicovitch sets  - sets containing a rotated (and translated) copy of $\Ga$ in every direction - of zero Lebesgue measure. 
In \cite{CC}, A. Chang and the second author showed that for any rectifiable planar curve $\Ga$, there exists an ``almost-$\Ga$-Besicovitch set" - a set containing a rotated (and translated) copy of a full $\hau^1$-measure subset of $\Ga$ in each direction - of zero Lebesgue measure. In fact, they showed that $\Ga$ can be continuously turned  around covering zero Lebesgue measure, if at each time moment, we are allowed to remove a set of $\hau^1$-null measure of $\Ga$. For closely related results about the problem of moving curves continuously while covering arbitrarily small measure, see \cite{CHL} and \cite{HL}. 

The existence of the above mentioned ``almost-$\Ga$-Besicovitch sets" of measure zero for rectifiable curves $\Ga$ suggests that it is interesting to consider the Hausdorff dimension of $\Ga$-Besicovitch sets (or of ``almost-$\Ga$-Besicovitch sets" - one naturally expects that in terms of dimension, there is no difference between the two). It turns out that for sufficiently smooth curves, the answer is very satisfactory. As observed by A. Chang, it follows from the results of J. Zahl in \cite{Z} that for any $C^{\infty}$ curve $\Ga$ that is not a circular arc, any $\Ga$-Besicovitch set must have Hausdorff dimension 2. 
In fact, the same holds under the relaxed assumption that $\Ga$ is a $C^3$ curve, see \cite{PYZ}.
Curvature plays an important role here.
Another interesting fact, which has been mentioned in \cite{CC}, is that if $\Ga$ is a countable union of (not necessarily concentric) circles, then there exists a $\Ga$-Besicovitch set of Hausdorff dimension $1$. The next natural question could be: what happens for non-smooth curves that behave differently from circular arcs? In fact, it would be interesting to consider objects whose behaviour is much closer to the behaviour of line segments than of circular arcs. Two immediately visible reasons for the special role of circular arcs are rotational symmetry and constant curvature. There is a third, probably most relevant aspect: the geometry of normal lines. 
As a special case of the results in \cite{CC}, it is known that if $\Ga$ is a rectifiable set possessing a continuous tangent field, then $\Ga$ can be continuously turned  around covering zero Lebesgue measure, if at each time moment $t$, we are allowed to choose a line $\ell_t$ and delete the points of $\Ga$ at which the normal line is equal to $\ell_t$. When $\Ga$ is a line segment, all of its normal lines are parallel, while in case $\Ga$ is a cirular arc, all of its normal lines point in different directions. Based on this, it would be interesting to consider rectifiable curves that, at almost every point, have parallel normal lines. 

In this paper, as a first attempt in this direction, we consider Besicovitch sets of Cantor graphs. More precisely, we consider a large class of (nearly self-similar) Cantor sets $C$ (see Section \ref{caset}), and build  rectifiable sets $\Ga$ on $C$ that resemble  Cantor-graphs (see Section \ref{cagra}). This is a generalization of the usual Cantor function, a.k.a. Devil's staircase built on the standard middle third Cantor set, noting that naturally, we only consider the Cantor-part of the graph - the horizontal line segments are not included. 
We show that the Hausdorff dimension of $\Ga$-Besicovitch sets must be at least $\min\left(2-s^2,\frac{1}{s}\right)$, where $s=\dim C$. 
As it turns out, it is important that our objects possess some structure. 
In an upcoming paper, see \cite{A}, the first author shows that for rectifiable graphs $\Ga$ built on general Cantor-type sets, the corresponding $\Ga$-Besicovitch set can even have dimension one.

The paper is organized as follows. In Section \ref{cant} we introduce the main setting: we define Besicovitch sets of general Cantor graphs built on Cantor sets, and we state our main result, Theorem \ref{dimb}. Section \ref{main} contains the main body of our proof: the proof of Proposition \ref{areas}. 
Finally, in Section \ref{pigeon} we include a standard argument that deduces the Hausdorff dimension bound of Theorem \ref{dimb} from Proposition \ref{areas}.

\subsection{Notation}
The 2-dimensional Lebesgue measure is denoted by $\leb$, and the 1-dimensional Lebesgue measure by $\la$. 
The open ball of center $x$ and radius $r$ is denoted by $B(x,r)$.
For a set $U \su \rr^n$, $U(\de)=\bigcup_{x \in U} B(x,\de)$ denotes the open $\de$-neighborhood of $U$, and $\diam (U)$ denotes the diameter of $U$. 
For any $s \geq  0$, $\de \in (0,\infty]$, the $s$-dimensional Hausdorff $\de$-premeasure is denoted by $\hau^s_{\de}$, and the 
$s$-dimensional Hausdorff measure is denoted by $\hau^s$. We use the notation $\dim A$ for the Hausdorff dimension of $A \su \rr^n$. 
For the well known properties of Hausdorff measures and dimension, see e.g. \cite{Fa}.

For a finite set $A$, let $|A|$ denote its cardinality. 
We write $x \lkb y$ for  $x \leq Cy$ where $C$ is an absolute constant or a constant depending only on $a$ and $b$, where $a, b$ are fixed constants, see the the next section.  We write $x \approx y$ if $x  \lkb y$ and $y  \lkb x$. 
Let $\prx$, $\pry$ denote the projections onto the $x$ and $y$ coordinate axes, respectively. 

\section{The setup and our main result}
\label{cant}
\subsection{Cantor sets}
\label{caset}
We describe a general Cantor set construction, where the ratio of removed intervals will be fixed throughout the process. 
Let $a \in \nn, a \geq 3$, $b \in \nn, 2 \leq b < a$, and let $J=\{d_0, d_1, \dots, d_{b-1}\} \su \{0,1,\dots a-1\}$ be a digit set such that $0 \leq d_0 \leq d_1 \leq \dots \leq d_{b-1} \leq a-1$.  
Partition $[0,1]$ into $a$ closed intervals $I_0,\dots, I_{a-1}$ of equal length (denoted this way from left to right). We keep $I_{d_0},\dots, I_{d_{b-1}}$, and remove the interiors of the remaining $a-b$ intervals. Let $C_1=I_{d_0} \cup I_{d_1} \cup \cdots \cup I_{d_{b-1}}$. We will iterate this process below to define $C_n$, $n=1,2,\dots$, such that the next digit set (of cardinality $b$) will be chosen independently in each subinterval in each step. 

For each $x_1 \in J$, let  $J_{(x_1)} \su \{0,1,\dots a-1\}$ with $|J_{(x_1)}|=b$. Whenever $(x_1\dots x_n)$ is a sequence of length $n$ such that $J_{(x_1\dots x_n)}$ is already defined 
and $x_{n+1} \in J_{(x_1\dots x_n)}$, let $J_{(x_1\dots x_n x_{n+1})}  \su \{0,1,\dots a-1\}$ with $|J_{(x_1\dots x_n x_{n+1})}|=b$. 
For notational convenience we wright $J=J_{(x_1x_0)}$.
Let 
$$C_n=\bigcup_{\substack{x_j \in J_{(x_1\dots x_{j-1})}, \\ j=1,\dots,n}} \left[\sum_{j=1}^n \frac{x_j}{a^j}, \sum_{j=1}^n \frac{x_j}{a^j} + \frac{1}{a^n}\right].$$ 
By construction, $C_n$ is a union of $b^n$ intervals of length $1/a^n$, and the lengths of the gaps between the intervals in $C_n$ are integer multiples of $1/a^n$. Let $C=\bigcap_{n=1}^\infty C_n$, our Cantor set. By construction, 
$C=\left\{\sum_{j=1}^\infty \frac{x_j}{a^j}: x_j \in J_{(x_1\dots x_{j-1})} \ \forall \ j \right\}$.  
Let $s=\log b / \log a=\dim C$. In the special case when each $J_{(x_1\dots x_n)}=J$, $C$ is a self-similar set.

\subsection{Cantor graphs}
\label{cagra}
Now we define a general Cantor graph construction built on $C$. Whenever $J_{(x_1\dots x_n)}$ is defined, let 
$\sigma_{(x_1\dots x_n)}: J_{(x_1\dots x_n)} \to \{0,1,\dots,b-1\}$ be an arbitrary bijection. 
Define $$
\Ga=\left\{\left(\sum_{j=1}^\infty \frac{x_j}{a^j},\sum_{j=1}^\infty \frac{\si_{(x_1\dots x_{j-1})}(x_j)}{b^j}\right): x_j \in J_{(x_1\dots x_{j-1})} \ \forall \ j \right\}.
$$
Note that in the special case when each $\sigma_{(x_1\dots x_n)}$ is order preserving, $\Ga$ is the graph of a monotone function, in fact, it is the classical devil's staircase graph built on $C$. 
We also note that in some special cases $\Ga$ is not precisely a graph, since at the points $x$ where the representation $x=\sum_{j=1}^\infty \frac{x_j}{a^j}$ is not unique we might have two points above $x$, but this is irrelevant for our purposes.

We list some well-known or easily verifiable properties of our Cantor graphs. 
\begin{enumerate}[i)]

\item 
$\dim \Ga=1$, and $\hau^1 \restriction \Ga=\la  \restriction \pry (\Ga)$.

\item
$\Ga$ is rectifiable, and for $\hau^1$-almost all $p \in \Ga$, the tangent to $\Ga$ at $p$ is vertical. 
\end{enumerate}

We will consider the natural approximation to $\Ga$ consisting of  $b^n$ rectangles of size $\frac{1}{a^n} \times \frac{1}{b^n}$: 
\begin{align*}
\bigcup_{\substack{x_j \in J_{(x_1\dots x_{j-1})}, \\ j=1,\dots,n}} & \left[\sum_{j=1}^n \frac{x_j}{a^j}, \sum_{j=1}^n \frac{x_j}{a^j} + \frac{1}{a^n}\right] \times \\ 
& \left[\sum_{j=1}^n \frac{\si_{(x_1\dots x_{j-1})}(x_j)}{b^j}, \sum_{j=1}^n \frac{\si_{(x_1\dots x_{j-1})}(x_j)}{b^j}+\frac{1}{b^n}\right].
\end{align*}
For technical reasons, we will also consider the 3 times larger rectangles: 
\begin{align}
\label{eqapp}
 R_n = 
\bigcup_{\substack{x_j \in J_{(x_1\dots x_{j-1})}, \\ j=1,\dots,n}} & \left[\sum_{j=1}^n \frac{x_j}{a^j}-\frac{1}{a^n}, \sum_{j=1}^n \frac{x_j}{a^j} + \frac{2}{a^n}\right] \times \\ 
& \left[\sum_{j=1}^n \frac{\si_{(x_1\dots x_{j-1})}(x_j)}{b^j}-\frac{1}{b^n}, \sum_{j=1}^n \frac{\si_{(x_1\dots x_{j-1})}(x_j)}{b^j}+\frac{2}{b^n}\right]. \nonumber
\end{align}

\subsection{Besicovitch sets of Cantor graphs}

We say that $E \su \rr^2$ is a $\Ga$-Besicovitch set, if $E$ contains a rotated (and translated) copy of $\Ga$ in every direction. 
That is, for every $\te \in  [0,\pi]$ there exists $\om_\te \in \rr^2$ such that $E=\bigcup_{\te \in  [0,\pi]} \tau_\te(\Ga)$, where 
$\tau_\te(z)=e^{i\te}z+\om_\te$ for $z \in \cc$. Here we used the identification $z=(z_1,z_2) \leftrightarrow z=z_1+iz_2$ between $\rr^2 $ and $\cc$ 
which will be used occasionally in the paper for convenience. 
Denote $\Ga_\te=\tau_\te(\Ga)$. Recall that $s=\dim C=\log b / \log a$. 
Our main result is the following.

\begin{theorem}
\label{dimb} 
Let $E \su \rr^2$ be a Borel $\Ga$-Besicovitch set. Then 
$$\dim E \geq \min\left(2-s^2,\frac{1}{s}\right).
$$
\end{theorem}

\section{Main proof}
\label{main}

Let $\de=\frac{1}{a^n}$ for some $n \in \nn$, and write $R=R_n=\bigcup_{i=1}^{b^n} T_i$ for the union of renctagles defined in \eqref{eqapp},  where the rectangles are indexed using their vertical order. That is, $\min \pry(T_1) < \min \pry(T_2) < \dots < \min \pry(T_{b^n})$, and $T_i$ is a rectangle of size $3\de \times 3\rde$ for each $i$. 
By construction, it is easy to see that $\Ga(\de) \su R$, and that $\leb(R)\leq 
9\de$. 
Write $R_\te=\tau_\te(R)$.

First we show that each fixed rectangle of $R$ can intersect at most 10 rectangles of $R_\te$. 
\begin{lemma}
\label{simple1}
Let $\te \in   [0,\pi]$ and $i \in \left\{1,\dots,b^n\right\}$. Then 
$$|\left\{j \in \{1,\dots,b^n\}: T_i \cap T_{j,\te} \neq \emptyset \right\}| \leq 10.$$
\end{lemma}

\begin{proof}
We have $\diam T_i=3\sqrt{\de^2+\de^{2s}} < 6 \rde$. Let $j,j' \in \left\{1,\dots,b^n\right\}$ with $j'-j>10$. 
Since $\min(\pry(T_{j+1}))-\min(\pry(T_j))=\rde$ and $\la(\pry(T_j))=3\rde$,
we get that 
$$d(T_{j,\te},T_{j',\te})=d(T_j,T_{j'}) \geq 10\rde-3\rde > \diam T_i.$$
This means that $T_{j,\te}$ and $T_{j',\te}$ can not  simultaneously intersect $T_i$. 
Therefore, if $T_{j-k,\te}, T_{j-k+1,\te}, \dots, T_{j,\te},\dots, T_{j+l-1,\te}, T_{j+l,\te}$ all intersect $T_i$, then we must have $j+l-(j-k)=l+k \leq 10$, and so
$$|\left\{j \in \{1,\dots,b^n\}: T_i \cap T_{j,\te} \neq \emptyset \right\}| \leq 10.$$
\end{proof}

We introduce a notation for the number of intersecting pairs of rectangles of $R$ and $R_\te$, let 
\begin{equation}
\label{ldete}
L(\de,\te)=|\{(i,j) \in \left\{1,\dots,b^n\right\}: T_i \cap T_{j,\te} \neq \emptyset\}|. 
\end{equation}
We note that by Lemma \ref{simple1}, $L(\de,\te) \lkb |\{i: \ \text{there is} \ j \ \text{such that} \ T_i \cap T_{j,\te} \neq \emptyset\}|$, and the constant in $\lkb$ is independent of $\de$ and $\te$. Therefore, we trivially have $L(\de,\te) \lkb \de^{-s}$, which is the number of all the rectangles in $R$. 

We also need the following very easy observation about intersecting rectangles. 
\begin{lemma}
\label{int}
For any $\te \in  (0,1],i,j$, we have $\leb(T_i \cap T_{j,\te}) \lkb \de \cdot \rde$, and $\leb(T_i \cap T_{j,\te}) \lkb \frac{\de^2}{\te}$.
\end{lemma}

\begin{proof}
Recall that $T_i$ and $T_{j,\te}$ are $3\de \times 3\rde$ rectangles with angle $\te$ between their longer side, the first inequality is trivial as $\leb(T_i) \lkb \de \cdot \rde$. 
The second statement is also trivial noting that the area of the intersection of two infinite strips of width $\de$ with angle $\te$ is at most $\frac{\de^2}{\sin \te}\lkb \frac{\de^2}{\te}$. 
\end{proof}

\begin{remark}
\label{remareas}
We have that $\frac{\de^2}{\te} > \de \cdot \rde$ if and only if $\te < \de^{1-s}$, so the first bound in Lemma \ref{int}, although trivial, is actually useful for some angles. 

\end{remark}

Now we let $\de_0 \leq 1$ be small enough (to be specified later), and fix $\de \leq \de_0$. Without loss of generality, we can assume that we only have rotated copies of $\Ga$ with angle in $[0,1]$. Let $A$ be a maximal $\de$-separated subset of $ [0,1]$, it is easy to see that $|A| \approx \frac{1}{\de}$. 

The main result of this section is the following. 

\begin{prop}
\label{areas}
For any $\de \leq \de_0$, we have 
\begin{equation}
\label{eqc0}
\sum\limits_{\te,\te' \in A} \leb(R_\te \cap R_{\te'}) \lkb \max\left(\de^{-s^2}, \de^{1/s-2}\right) \cdot \log\left(\frac{1}{\de}\right).
\end{equation}
\end{prop}
Using the Cauchy-Schwarz inequality, Proposition \ref{areas} will easily yield the bound of Theorem \ref{dimb} for (lower) Minkowski dimension, we include the details below. A standard pigeonholing argument can be used to prove the corresponding Hausdorff dimension bound, this will be carried out in Section \ref{pigeon}.

\begin{proof}[Proof of Theorem \ref{dimb}  for (lower) Minkowski dimension ]
Let $d=\min\left(2-s^2,\frac{1}{s}\right)$. Using $\bigcup_{\te \in A} \Ga_\te(\de) \su E(\de)$ and $\Ga_\te(\de) \su R_\te$ for each $\te \in A$, we have 
\begin{align}
\label{eqc2}
\sum_{\te \in A} \leb(\Ga_\te(\de) ) & \leq \leb\left(\bigcup_{\te \in A} \Ga_\te(\de) \right)^{1/2} \cdot 
\left(\sum_{\te, \te' \in A} \leb(\Ga_{\te'}(\de)  \cap\Ga_\te(\de) )\right)^{1/2} \leq \\
& \leq \leb\left(E(\de)\right)^{1/2} \cdot \left(\sum_{\te, \te' \in A}  \leb(R_{\te'} \cap R_\te)\right)^{1/2}. \nonumber 
\end{align}
For the left hand side, it is easy to see that $\sum\limits_{\te \in A} \leb\left(\Ga_\te(\de)\right)  \gkb \frac{1}{\de} \cdot \de =1$, 
and by Proposition \ref{areas}, $\sum\limits_{\te,\te' \in A} \leb(R_\te \cap R_{\te'}) \lkb \de^{d-2}\cdot \log(\frac{1}{\de}).$ 
This yields $\leb\left(E(\de)\right) \gkb \frac{\de^{2-d}}{\log(1/\de)}$, therefore the Minkowski dimension of $E$ is at least $d$. 
\end{proof}

We now turn to the proof of  Proposition \ref{areas}.
\begin{proof}[Proof of Proposition \ref{areas}]

We have
$$
\sum\limits_{\te,\te' \in A} \leb(R_\te \cap R_{\te'})=
 \sum\limits_{\te,\te' \in A, |\te-\te'| \geq \de} \leb(R_\te \cap R_{\te'})+
\sum\limits_{\te' \in A} \leb(R_{\te'}).$$
For the second term, we get 
\begin{equation}
\label{trivi}
\sum\limits_{\te' \in A} \leb(R_{\te'}) \lkb \frac{1}{\de}\cdot \de = 1.
\end{equation}
Now we fix $\te'$, and for simplicity, we assume without loss of generality that $\te'=0$ and $R_{\te'}=R$.
So we need to give an upper bound for 
\begin{equation}
\label{eqbase}
 \sum\limits_{\te \in A, \te \geq \de} \leb(R_\te \cap R).
\end{equation}

\begin{remark}
\label{remeasy}
Using the observation above Lemma \ref{int}, we can give an easy upper bound for \eqref{eqbase}. Namely, we get that 
\begin{equation}
\label{eqeas}
\sum\limits_{\te \in A, \te  \geq \de} \leb(R_\te \cap R) \lkb  
\sum\limits_{\te \in A, \te  \geq \de} L(\de,\te) \cdot \frac{\de^2}{\te} \lkb \de^{1-s} \log\left(\frac{1}{\de}\right). 
\end{equation}
This yields 
$$\sum\limits_{\te,\te' \in A} \leb(R_\te \cap R_{\te'}) \lkb \de^{-s}\cdot \log\left(\frac{1}{\de}\right).$$
Using the same standard argument as in the proof of Theorem \ref{dimb}, it follows that the Hausdorff dimension of $\Ga$-Besicovitch sets is at least $2-s$.  
\end{remark}

To improve this easy bound, we will improve the bound in \eqref{eqeas} to 
\begin{equation}
\label{eqbass}
\sum\limits_{\te \in A, \te \geq \de} \leb(R_\te \cap R) \lkb \max\left(\de^{1-s^2}, \de^{1/s-1}\right) \cdot \log\left(\frac{1}{\de}\right).
\end{equation}
In order to do this,  we will consider two substantially different cases: the small angle and the large angle cases. 
We say that 
\begin{equation}
\label{smallte}
\te \ \text{is} 
\begin{cases}
  \text{small if} \ &  \de \leq \te \leq \min\left(\de^s,\de^{1-s}\right)=:\be(\de)=:\be,  \\ 
 \text{large if} \ &  \be \leq \te, 
\end{cases}
\end{equation}

\begin{equation}
\label{velate}
\te \ \text{is} \ \text{very large if} \  \ga:=\ga(\de):=\max\left(\de^s,\de^{1-s}\right) \leq \te. 
\end{equation}
We have the following estimates for the number of intersecting pairs of rectangles. 
\begin{lemma}
\label{sangle}
If $\te$ is small, then 
$$L(\de,\te) \lkb  \frac{1}{\de^s} \cdot   \max\left(\frac{\de}{\te},\te^s\right).$$
\end{lemma}

\begin{lemma}
\label{langle}
If $\te$ is large, then 
$$L(\de,\te) \lkb 
\frac{1}{\de^s} \cdot \max\left(\frac{r_0}{\te},\te^s\right), 
$$
moreover, if $\te \geq \de^{1-s}$, then 
$$L(\de,\te) \lkb
  \frac{1}{\te^{s}} \cdot \frac{1}{\de^{s^2}} \cdot \max\left(\frac{r_0}{\te}, \te^s\right),$$ 
where $r_0=\max\left(\te^{1/s}, \te^{1/(1-s)} \right)$.
\end{lemma}

 We start with the proof of Lemma \ref{sangle}. 

\begin{proof}[Proof of Lemma \ref{sangle}]
Fix $\te$ such that $\te \in A, \de \leq \te \leq \be$, and choose the unique $m \in \{0,\dots,n\}$ such that $\te \in (1/a^{m+1},1/a^m]$.
Recall that $R=\bigcup_{i=1}^{b^n} T_i, R_\te=\bigcup_{i=1}^{b^n} T_{i,\te}$. 
Fix $T_i$ and $T_{j,\te}$ such that they intersect. For simplicity, let $x$,$y$ denote the bottom left corners of $T_i$ and $T_j$, respectively, and let $y_\te$ be the bottom left corner of $T_{j,\te}$. We will use the following simple geometric lemma. 
\begin{lemma}
\label{simple2}
Let $\te \in (0,1]$ and  assume that $T_i \cap T_{j,\te} \neq \emptyset$. Then 
\begin{enumerate}
	\item $|\pry(x-y_\te)|\leq 10 \rde$.
	
	\item If $\te \leq \de^{1-s}$, then $|\prx(x-y_\te)|\leq 10 \de$.
\end{enumerate}

\end{lemma}
\begin{proof}[Proof of Lemma  \ref{simple2}]
Clearly, 
$|\pry(x-y_\te)| \leq \la(\pry(T_i))+\la(\pry(T_{j,\te}))=3\rde+\la(\pry(T_{j,\te}))$. 
Similarly, 
$|\prx(x-y_\te)| \leq \la(\prx(T_i))+\la(\prx(T_{j,\te}))=3\de+\la(\prx(T_{j,\te}))$. 
By simple planar geometry, we have that 
$\la(\pry(T_{j,\te})) \leq 3\rde|\cos{\te}|+3\de|\sin{\te}| \leq 6\rde$. 
Similarly, if $\te \leq \de^{1-s}$, then 
$\la(\prx(T_{j,\te})) \leq 3\rde|\sin{\te}|+3\de|\cos{\te}| \leq 3\rde \te+3\de \leq 6\de$ and the claim of Lemma \ref{simple2} follows. 

\end{proof}

By Lemma \ref{simple2}, since $\te$ is small, see \eqref{smallte}, we have 
\begin{equation}
|\prx(x-y_\te)|\leq 10 \de, \ |\pry(x-y_\te)|\leq 10 \rde. 
\label{eqs2}
\end{equation}
For $x=\left(\sum\limits_{j=1}^n \frac{x_j}{a^j}-\frac{1}{a^n}, \sum\limits_{j=1}^n \frac{\si_{(x_1\dots x_{j-1})}(x_j)}{b^j}-\frac{1}{b^n}\right)$, we write 
$x=x_{\lar}+x_{\sma}$, where 
\begin{align}
&  x_{\lar}=\left(\sum_{j=1}^m \frac{x_j}{a^j}, \sum_{j=1}^m \frac{\si_{(x_1\dots x_{j-1})}(x_j)}{b^j}\right), \\ 
& x_{\sma}=\left(\sum_{j=m+1}^n \frac{x_j}{a^j}-\frac{1}{a^n}, \sum_{j=m+1}^n \frac{\si_{(x_1\dots x_{j-1})}(x_j)}{b^j}-\frac{1}{b^n}\right). \nonumber
\label{eqs3}
\end{align}
Similarly, write $y=y_{\lar}+y_{\sma}$. 
Note that by construction, 

\begin{equation}
\label{props1}
\pry (x_{\lar}),  \pry (y_{\lar}) \ \text{are integer multiples of} \ 1/b^m, 
\end{equation}

\begin{equation}
\label{props2}
\pry (x_{\sma}),  \pry (y_{\sma}) \ \text{are integer multiples of} \ 1/b^n=\rde,
\end{equation}

\begin{equation}
\label{props3}
\prx (x_{\lar}),  \prx (y_{\lar}) \ \text{are integer multiples of}\ 1/a^m,
\end{equation}

\begin{equation}
\label{props4}
\prx (x_{\sma}),  \prx (y_{\sma}) \ \text{are integer multiples of}\ 1/a^n=\de.
\end{equation}

Now we fix $\xi \in \left\{x_{\sma}: x_{m+1},\dots,x_n \in J\right\}$, we have $b^{n-m}$ choices for $\xi$. 
Let $M_\te$ denote the number of pairs $(x,y)$ 
such that the corresponding rectangles $T_i$ and $T_{j,\te}$ intersect and such that $x_{\sma}=\xi$. Clearly, 
\begin{equation}
\label{lbound}
L(\de,\te) \lkb b^{n-m} \cdot M_\te \lkb \frac{\te^s}{\de^s} \cdot M_\te. 
\end{equation}

\begin{claim}
\label{mtheta}
We have 
$$M_\te \lkb \max \left(\frac{\de}{\te^{1+s}},1 \right).$$
\end{claim}

\begin{proof}[Proof of Claim \ref{mtheta}]
Suppose now that  $(x,y)$ satisfy \eqref{eqs2} and that $x_{\sma}=\xi$.
We will need the following simple geometric lemma: 
\begin{lemma}
\label{simple3}
For any $z \in \cc$, we have 
\begin{enumerate}
\item
\label{s31}
$|\prx(e^{i\te}z)-\prx z| \leq |\te z|$.

\item
\label{s32}
$|\pry(e^{i\te}z)-\pry z| \leq |\te z|$.

\item
\label{s33}
$||\prx(e^{i\te}z)-\prx z|-|\te \pry z|| \leq 2|\te^2 z|$.

\end{enumerate}
 
\end{lemma}

\begin{proof}
The simple proof is left to the reader. 

\end{proof}

By \eqref{eqs2}, we have 
$|\pry y_\te-\pry x| \leq 10 \rde$.
By \eqref{s32} of Lemma \ref{simple3} applied for $y$ and using that $|y| \leq 2$ (as $y \in R \su [-\de,1+\de] \times [-\rde,1+\rde])$ as well as $y_\te=e^{i\te}y+\om_{\te}$, we have 
$|\pry(e^{i\te}y)-\pry y|\leq 2\te$, and so 
$$|\pry y-\pry x+\pry(\om_{\te})|= |\pry y- \pry(e^{i\te}y)+\pry y_\te -\pry x|\leq 2 \te+10 \rde.$$
Using $\te \leq \be \leq \rde$, 
we get that 
$$|\pry y-\pry x+\pry(\om_{\te})| \leq 12\rde.$$
By \eqref{props1}, $\pry (y_{\sma})\equiv \pry y \pmod{1/b^m}$ and   
$\pry (x_{\sma})\equiv\pry x  \pmod{1/b^m}$. Therefore, since $x_{\sma}=\xi$, we have 
\begin{align}
\label{eqspi}
\pry (y_{\sma}) \in & \left(\pry x -\pry(\om_{\te}) + \left[-12 \rde,12 \rde\right] \pmod{1/b^m} \right)=  \nonumber \\
& \left(\pry \xi-\pry(\om_{\te}) + \left[-12 \rde, 12 \rde \right] \pmod{1/b^m}\right).
\end{align}
Using \eqref{props2} and the fact that $\pry (y_{\sma})\in [0,1/b^m)-1/b^n$, \eqref{eqspi} implies that 
$\pry (y_{\sma})$ can only take at most $25$ different values. 
Consequently, since by construction,  $\pry (y_{\sma})$ uniquely determines $y_{\sma}$, 
\begin{equation}
\label{eqs4}
y_{\sma} \ \text{can only take at most} \  25 \ \text{different values}. 
\end{equation}
Fix $z$ to be one of these $25$ different values. 

Suppose now that $(x,y)$, and also $(x',y')$ satisfy \eqref{eqs2} and that $x_{\sma}=\xi=x'_{\sma}$, $y_{\sma}=z=y'_{\sma}$. 
By \eqref{eqs2}, we have 
$|\prx(x-x'-y_\te+y'_\te)|\leq 20 \de$. We also have 
$$x-x'-y_\te+y'_\te=x-x'-e^{i\te}y+e^{i\te}y'=x-x'-e^{i\te}(y-y').$$ 
By \eqref{s33} of Lemma \ref{simple3} applied to $y-y'$, we get 
$$||\prx(e^{i\te}(y-y'))-\prx (y-y')|-|\te \pry (y-y')||\leq 2\te^2|y-y'|.$$
Since
$$
|\prx(x-x'-y+y')-\prx(x-x'-y_\te+y'_\te)|=|\prx(e^{i\te}(y-y'))-\prx (y-y')|,
 $$
we also get 
\begin{equation}
\label{eqs7}
||\prx(x-x'-y+y')-\prx(x-x'-y_\te+y'_\te)|-|\te \pry (y-y')||\leq 2\te^2|y-y'|.
\end{equation}
Since $|\prx(x-x'-y_\te+y'_\te)|\leq 20\de \leq 20/ a^m$, $\te|\pry  (y-y')| \leq \te \leq 1/ a^m$, and $2\te^2|y-y'|\leq 4\te \leq 4/ a^m$, we must have 
$|\prx(x-x'-y+y')| \leq   25/ a^m$.

Now we utilize 
$\prx(x-x'-y+y')=\prx(x-x'-y+y')_{\lar}+\prx(x-x'-y+y')_{\sma}.$
Since by definition, $y_{\sma}=y'_{\sma}$ and $x_{\sma}=x'_{\sma}$, we have 
$$\prx(x-x'-y+y')_{\sma}=0.$$ 
Therefore, 
$$|\prx(x-x'-y+y')_{\lar}|=|\prx(x-x'-y+y')| \leq   25/ a^m.$$
Finally, since by \eqref{props3},  $\prx(x-x'-y+y')_{\lar}$ is an integer multiple of $1/a^m$, we have that

\begin{equation}
\label{eqs6}
\prx(x-x'-y+y') \ \text{can only take at most} \  51 \ \text{different values}.
\end{equation}
Now, using \eqref{eqs7}, $|\prx(x-x'-y_\te+y'_\te)| \leq 20\de$, $\te^2 \leq \de^s \de^{1-s}= \de$, and $|y-y'| \leq 2$, we get that 
\begin{align*}
|\te \pry (y-y')| & \leq |\prx(x-x'-y+y')|+|\prx(x-x'-y_\te+y'_\te)|+ 2\te^2|y-y'|) \leq \\ 
& \leq |\prx(x-x'-y+y')| + 24 \de.
\end{align*}
Similarly, 
\begin{align*}
|\te \pry (y-y')| & \geq |\prx(x-x'-y+y')|-|\prx(x-x'-y_\te+y'_\te)|- 2\te^2|y-y'| \geq \\
& \geq |\prx(x-x'-y+y')|-24\de.
\end{align*}
Combining these with \eqref{eqs6}, we obtain that $|\te \pry (y-y')|$ is contained in the union of at most $51$ intervals of length $48\de$.
By construction, since $y_{\sma}=y'_{\sma}$, 
$\pry (y-y')=\frac{l}{b^m}$ for some integer $l$. We get that $l$ is contained in the union of at most $51$ intervals of length  
$\frac{48 \de \cdot b^m}{ \te} \lkb \frac{\de}{\te^{s+1}}$.
Since $l$ is an integer, this means that the number of possible values of $\pry (y-y')$ is $\lkb \max\left(\frac{\de}{\te^{s+1}},1\right)$. 
Since $\pry (y)$ uniquely determines $y$, this means that the number of $y$ such that $y_{\sma}=z$ and \eqref{eqs2} holds for $y$ and some $x$, is 
$\lkb \max\left(\frac{\de}{\te^{s+1}},1\right)$. Taking all $25$ values of $z$ and using Lemma \ref{simple1}, by the definition of $M_\te$ we get 
$$
M_\te \lkb \max \left(\frac{\de}{\te^{s+1}},1\right)
$$
and the proof of Claim \ref{mtheta} concludes. 
\end{proof}

Using \eqref{lbound} and Claim \ref{mtheta}, 
we get that 
\begin{equation}
\label{eqind}
L(\de,\te) \lkb \frac{\te^s}{\de^s}  \cdot M_\te \lkb  \frac{\te^s}{\de^s}  \cdot  \max \left(\frac{\de}{\te^{s+1}},1\right) =\frac{1}{\de^s} \cdot   \max\left(\frac{\de}{\te},\te^s\right),
\end{equation}
and the proof of Lemma \ref{sangle} concludes.

\end{proof}

\begin{proof}[Proof of Lemma \ref{langle}]

Fix $\te$ such that $\te \in A, \be < \te \leq 1$.
For clarity, we will consider the $s \leq 1/2$ and $s > 1/2$ cases separately. 
Note that 
\begin{equation}
\label{bedef}
\be=\begin{cases}
\de^s \ & \text{when} \ s \geq 1/2 \\ 
\de^{1-s} \ & \text{when} \ s \leq 1/2. 
\end{cases}
\end{equation}
Let 
\begin{equation}
\label{r0def}
r_0=\begin{cases}
\te^{1/s} \ & \text{if} \ s \geq 1/2 \\ 
\te^{1/(1-s)} \ & \text{if} \ s \leq 1/2. 
\end{cases}
\end{equation}
Choose the unique $t \in \nn$ such that 
$r_0 \in (1/a^{t+1},1/a^t]$, and let $r=1/a^t$. We have 
\begin{equation}
\label{eqs93}
\te = \min\left(r_0^s,r_0^{1-s}\right)\leq \min\left(r^s,r^{1-s}\right)=:\be(r). 
\end{equation}
The main idea is the following: 
We will apply the previous "small angle" case for $r$ in place of $\de$, that is, for the approximation of $\Ga$ consisting of rectangles of size $3r \times 3r^s$ (note that $r$ is larger than $\de$). Moreover, we will show that inside each intersecting pair of rectangles of size $3r \times 3r^s$, there can not be too many intersecting rectangles of size $3\de \times 3\rde$. 

First, using Lemma \ref{sangle} for $r$ in place of $\de$ which is valid by \eqref{eqs93}, we obtain that 
\begin{equation}
\label{eqlarge1}
L(r,\te) \lkb \frac{1}{r^s} \cdot \max\left(\frac{r}{\te}, \te^s\right).
\end{equation}
Second, using that each $3r \times 3r^s$ rectangle contains $r^s/\de^{s}$ many $3\de \times 3\rde$ rectangles in our construction, we trivially have that 
\begin{equation}
\label{trivisc}
L(\de,\te) \lkb \frac{r^s}{\de^s}\cdot L(r,\te) \lkb  \frac{1}{\de^s}\cdot \max\left(\frac{r}{\te},\te^s\right) \lkb 
\frac{1}{\de^s}\cdot \max\left(\frac{r_0}{\te},\te^s\right), 
\end{equation}
which is the first claim of Lemma \ref{langle}. 
We will use this estimate when $s > 1/2$ and $\te$ is large but not very large, see \eqref{smallte} and \eqref{velate}. 

Assume now that $\te \geq \de^{1-s}$. This holds when $s \leq 1/2$ and $\te$ is large, and when $s > 1/2$ and $\te$ is very large. In this case, we will use the following lemma. 
\begin{lemma}
\label{cl1}
Let $\te \geq \de^{1-s}$. Then  
$L(\de,\te) \lkb \frac{r^s}{(\te \rde)^s} \cdot L(r,\te).$ 
\end{lemma}

\begin{proof}[Proof of Lemma \ref{cl1}]
We introduce another scale. Let  $\rho_0=\te \cdot \de^s$, choose the unique $k \in \nn$ such that $\rho_0 \in  (1/a^{k+1},1/a^k]$, and let $\rho=1/a^k$. 
One can easily check that $\de \leq \rho_0 \leq r_0$, and then $\de \leq \rho \leq r$. 
Recall that $T_i, T_{j,\te}$ denote rectangles of  size $3\de \times 3\rde$ in $R$, $R_\te$, respectively. 

\begin{lemma}
\label{lip1}
Let $S$ be a rectangle of $R_k$ and let $S'_\te$ be a rectangle of $R_{k,\te}$ such that $S \cap S'_\te \neq \emptyset$. 
Then $$|\left\{(i,j): T_i \su S,T_{j,\te} \su S'_\te, T_i \cap T_{j,\te} \neq \emptyset\right\}| \leq 220a.$$

\end{lemma}

\begin{proof}[Proof of Lemma \ref{lip1}]

We will start by showing the following claim. 
\begin{claim}
\label{lip2}
Assume that $i, i', j, j'$ such that $T_i \cap T_{j,\te} \neq \emptyset, T_{i'} \cap T_{j',\te} \neq \emptyset$, $T_i,T_{i'} \su S, T_{j,\te}, T_{j',\te} \su S'_\te$. Then $|i'-i| \leq 10a$ or $|j'-j| \leq 10a$. 
\end{claim}

\begin{proof}[Proof of Claim \ref{lip2}]
Let $v \in \rr^2$ be a vector such that $|\prx v | \leq 3\rho$, $|\pry v | \geq 8a \rde$. 
Let $\al$ denote the angle of $v$ and the vertical axis. We show that $\al < \te/2$: 
\begin{equation}
\label{eqs10}
\al \leq  \tan \al = \frac{|\prx v |}{|\pry v |}\leq \frac{3\rho}{8a\rde}\leq \frac{3a\rho_0}{8a\rde}=
\frac{3\te \cdot \de^s}{8\rde}< \frac{\te}{2}.
\end{equation}
Assume now for contradiction that  $|i'-i| > 10a$ and $|j'-j| > 10a$, $p \in T_i \cap T_{j,\te}$, $q \in T_{i'} \cap T_{j',\te}$, and let $v=p-q$. 
By construction, 
$|\pry v | \geq  (10a+1) \rde- 3 \rde \geq 8a \rde$, and since $p,q \in S$, we have $|\prx v | \leq 3\rho$. 
So by \eqref{eqs10}, $\al < \frac{\te}{2}$. 

Now let $\ell$ denote the vertical axis, $\ell_\te=e^{i\te}\ell$, let $\al_\te$ denote the angle of $v$ and $\ell_\te$, and let $\proj_{\ell_\te}, \proj_{\ell_{\te^{\perp}}}$ denote the orthogonal projection onto $\ell_\te$ and $\ell_{\te^{\perp}}$, respectively. 
By construction, similarly above, $|\proj_{\ell_\te}(v)| \geq 8a \rde$. Moreover, since $p,q \in S'_\te$, $|\proj_{\ell_{\te^{\perp}}}(v)| \leq 3 \rho$. 
Similarly as in \eqref{eqs10} above, we obtain that $\al_\te < \frac{\te}{2}$. 
But clearly, $\te=\al + \al_\te<\te/2+\te/2$, so we get a contradiction and the statement of Claim \ref{lip2} follows. 

\end{proof}
Let $i, i', j, j'$ as in the statement of Claim \ref{lip2}. 
Then  without loss of generality, $|i'-i| \leq 10a$. 
Therefore, the number of $T_i$'s contained in $S$ that intersect some $T_{j,\te} \su S'_\te$ is at most $22a$. 
By Lemma \ref{simple1}, we know that for each $T_i$ there are at most 10 $T_{j,\te}$ such that $ T_i \cap T_{j,\te} \neq \emptyset$. 
So we get that 
$$|\left\{(i,j): T_i \su S, T_{j,\te} \su S'_\te, T_i \cap T_{j,\te} \neq \emptyset\right\}| \leq 10 \cdot 22a=220a$$ 
and the proof of Lemma \ref{lip1} concludes.
\end{proof}

We are ready to finish the proof of Lemma \ref{cl1}. By Lemma \ref{lip1}, $L(\de,\te) \lkb  L(\rho,\te)$. Since the number of $3\rho \times 3\rho^s$-rectangles in our construction inside a fixed rectangle of size $3r \times 3r^s$ is $\lkb r^s/\rho^s$, clearly, we have 
$L(\rho,\te)  \lkb r^s/\rho^s  \cdot L(r,\te) \leq r^s/(\te \rde)^s  \cdot L(r,\te)$. Consequently, 
$L(\de,\te)  \lkb \frac{r^s}{(\te \rde)^s}  \cdot L(r,\te)$ and the proof of Lemma \ref{cl1} concludes. 

\end{proof}

Combining \eqref{eqlarge1} and Lemma \ref{cl1}, we obtain that 
\begin{equation}
\label{eqs94}
L(\de,\te) \lkb \frac{r^s}{(\te \rde)^s}  \cdot \frac{1}{r^s} \cdot \max\left(\frac{r}{\te}, \te^s\right)
\lkb  \frac{1}{\te^{s}} \cdot \frac{1}{\de^{s^2}} \cdot \max\left(\frac{r_0}{\te}, \te^s\right),
\end{equation}
which is the second statement in Lemma \ref{langle}. 

\end{proof}
 
Using the upper bounds for $L(\de,\te)$ obtained in Lemma \ref{sangle} and \ref{langle}, we now give an upper bound for $\leb(R_\te \cap R)$ and eventually for 
$\sum\limits_{\te \in A, \te  \geq \de} \leb(R_\te \cap R)$. 

Let $\te$ be small. Then by Lemma \ref{int} and Lemma \ref{sangle}, we have 
\begin{equation}
\label{eqsm1}
\leb(R_\te \cap R) \lkb \frac{1}{\de^s} \cdot   \max\left(\frac{\de}{\te},\te^s\right) \cdot \de^{s+1}=
 \de \cdot \max \left(\frac{\de}{\te},\te^s \right).
\end{equation}
Summing up for small angles $\te$, we obtain 
\begin{align*}
 \sum\limits_{\te \in A, \de \leq \te \leq \be}  \leb(R_\te \cap R) & \lkb \sum\limits_{\te \in A, \de \leq \te \leq \be} 
\de \cdot \max \left(\frac{\de}{\te},\te^s \right)= \\ 
& =\de^2 \cdot \sum\limits_{\te \in A, \de \leq \te \leq \de^{1/(s+1)}}  \frac{1}{\te} + 
\de \cdot  \sum\limits_{\te \in A, \de^{1/(s+1)} \leq \te \leq \be}  \te^s. 
\end{align*}
Easy computation gives that 
\begin{equation}
\label{eqs8}
\sum\limits_{\te \in A,  \de \leq \te \leq \de^{1/(s+1)}} \frac{1}{\te} \lkb \frac{1}{\de} \cdot \log\left(\frac{1}{\de}\right), \ 
\sum\limits_{\te \in A,  \de^{1/(s+1)} \leq \te \leq \be}  \te^s  \lkb \frac{1}{\de} \cdot \be^{s+1}. 
\end{equation}
Therefore, using the definition of $\be$, see \eqref{bedef}, easy computation gives that 
\begin{align*}
\sum\limits_{\te \in A, \de \leq \te \leq \be}  \leb(R_\te \cap R)\lkb 
\de \cdot  \log\left(\frac{1}{\de}\right) + \be^{s+1} \lkb \begin{cases}
\de \cdot  \log\left(\frac{1}{\de}\right) & \text{if} \ s \geq s_0, \\ 
\de^{s(s+1)} \cdot  \log\left(\frac{1}{\de}\right) & \text{if} \ \frac{1}{2} \leq s \leq s_0, \\
\de^{1-s^2} \cdot  \log\left(\frac{1}{\de}\right) & \text{if} \ s \leq \frac{1}{2}, 
\end{cases}
\end{align*}
where  $s_0=\frac{\sqrt{5}-1}{2}$. 
Since $s(s+1) \geq 1-s^2$ when $s \geq 1/2$, we obtain 
\begin{equation}
\label{eqs90} 
\sum\limits_{\te \in A, \de \leq \te \leq \be}  \leb(R_\te \cap R)\lkb 
\de^{1-s^2} \cdot  \log\left(\frac{1}{\de}\right)  
\end{equation}
for each $s$. 

Now we consider the large angles. 
Assume first that $s \leq 1/2$. In this case, $\te$ large implies that $\te \geq \de^{1-s}$, and $r_0= \te^{1/(1-s)}$, see \eqref{smallte}, \eqref{bedef}, and \eqref{r0def}.   
By Lemma \ref{int} and the second bound from Lemma \ref{langle} and using that $\te^{1/(1-s)} \leq \te^{s+1}$, we obtain 
$$
\leb(R_\te \cap R) \lkb \frac{1}{\de^{s^2}} \cdot \max\left(\frac{\te^{1/(1-s)}}{\te^{s+1}}, 1\right)
 \cdot \frac{\de^{2}}{\te}=\frac{\de^{2-s^2}}{\te}.$$ 
Summing up for $\te$, we get that 
$$\sum\limits_{\te \in A, \te  >\be} \leb(R_\te \cap R) \lkb \de^{2-s^2} 
\sum\limits_{\te \in A, \te  >\be} \frac{1}{\te} \lkb
\de^{1-s^2} \cdot \log\left(\frac{1}{\de}\right).$$
Now let $s>1/2$, so $r_0= \te^{1/s}$, see \eqref{r0def}. 
For $\te \geq \de^{1-s}$, we use the same bound from Lemma \ref{langle} as in the $s \leq 1/2$ case, and obtain that 
\begin{equation}
\label{final1}
\leb(R_\te \cap R) \lkb \frac{1}{\de^{s^2}} \cdot \max\left(\frac{\te^{1/s}}{\te^{s+1}}, 1\right)
 \cdot \frac{\de^{2}}{\te}= \frac{\de^{2-s^2}}{\te} \cdot \max\left(\frac{\te^{1/s}}{\te^{s+1}}, 1\right). 
\end{equation}
For $\be=\de^s \leq \te \leq \de^{1-s}$, we use the first bound from Lemma \ref{langle}, and the first trivial bound from Lemma \ref{int}. Therefore, 
\begin{equation}
\label{final2}
\leb(R_\te \cap R) \lkb \frac{1}{\de^s} \cdot \max\left(\frac{\te^{1/s}}{\te},\te^s\right) \cdot \de^{s+1} =
\de \cdot \te^s \cdot \max\left(\frac{\te^{1/s}}{\te^{s+1}}, 1\right)
\end{equation}
Easy computation gives that $\frac{\te^{1/s}}{\te^{s+1}} \leq 1$  if and only if $s \leq s_0=\frac{\sqrt{5}-1}{2}$. 

In the $s \leq s_0$ case, from \eqref{final1} and \eqref{final2} we obtain 
$$\leb(R_\te \cap R) \lkb 
\begin{cases}
\frac{\de^{2-s^2}}{\te} \ &  \text{for} \ \te \geq \de^{1-s}, \\
\de \cdot \te^s \ & \text{for} \ \de^s \leq \te \leq \de^{1-s},  
\end{cases}
$$
and then 
\begin{equation}
\label{final3}
\sum\limits_{\te \in A, \te  >\be} \leb(R_\te \cap R) \lkb \sum\limits_{\te \in A, \te  >\de^{1-s}}\frac{\de^{2-s^2}}{\te} + 
\sum\limits_{\te \in A, \de^s \leq \te \leq \de^{1-s}} \de \cdot \te^s \lkb
\de^{1-s^2}\log\left(\frac{1}{\de}\right).
\end{equation}
Combining \eqref{eqs90} with \eqref{final3}, we obtain the statement of \eqref{eqbass} for $s \leq s_0$.

In the $s \geq s_0$ case, we get 
$$\leb(R_\te \cap R) \lkb 
\begin{cases}
\frac{\de^{2-s^2}}{\te}\cdot \frac{\te^{1/s}}{\te^{s+1}} \ & \text{for} \ \te \geq \de^{1-s}, \\
\de \cdot \te^s \cdot \frac{\te^{1/s}}{\te^{s+1}} \ & \text{for} \ \de^s \leq \te \leq \de^{1-s}.   
\end{cases}
$$
Summing up for $\te$, easy computation using that $1/s-s-1 \leq 0$ as $s \geq s_0$ yields 
\begin{align}
\label{final4}
\sum\limits_{\te \in A, \te  >\be} & \leb(R_\te \cap R)  \lkb \\ 
& \sum\limits_{\te \in A, \te  >\de^{1-s}} \de^{2-s^2} \te^{1/s-s-2} + 
\sum\limits_{\te \in A, \de^s \leq \te \leq \de^{1-s}} \de \cdot \te^{1/s-1} \lkb   
 \de^{1/s-1}. \nonumber
\end{align}
Combining \eqref{eqs90} with \eqref{final4}, we obtain the statement of \eqref{eqbass} for $s \geq s_0$ as well. 

Finally, using \eqref{eqbass}, summing up for $\te' \in A$ as well as accounting for \eqref{trivi}, we get the statement of  Proposition \ref{areas}: 
$$\sum\limits_{\te,\te' \in A} \leb(R_\te \cap R_{\te'}) \lkb  \frac{1}{\de} \cdot \sum\limits_{\te \in A, \te \geq \de} \leb(R_\te \cap R) +1 \lkb
\max\left(\de^{-s^2}, \de^{1/s-2}\right) \cdot \log\left(\frac{1}{\de}\right).
$$

\end{proof}

\section{Proof of Theorem \ref{dimb} from Proposition \ref{areas}}
\label{pigeon}

The proof is based on a standard pigeonholing argument. For the sake of completeness, we include the proof. 
Let $d=\min\left(2-s^2,\frac{1}{s}\right)$, fix $0<\ep<d$, and write $u=d-\ep$. 
Let $K$ be a positive integer such that  $\sum_{l=K}^\infty 1/l^2 < 1$ and $2a^{-K} \leq \de_0$. 
Moreover, let $0<\eta \leq 2a^{-K}$. 
We will show that  $\hau^u_\eta(E) \gkb_\ep 1$ (independently of $\eta$), so $\dim E \geq d$. 

Let $E \su \bigcup_{i=1}^{\infty} B(x_i,r_i)$ be any countable cover of $E$ consisting of disks with $2 r_i \leq \eta \leq 2a^{-K}$ for all $i$.
For any $l \geq K$, let 
$$J_l=\left\{ i : a^{-l-1} < r_i \leq a^{-l} \right\}, \ \text{and let} \ E_l=\left(\bigcup_{i \in J_l} B(x_i,r_i)\right) \cap E.$$ 
Then $E=\bigcup_{l=K}^{\infty} E_l$, and $E_l$ is covered by disks of radius $\approx a^{-l}$.  
Using a pigeonholing argument, we prove the following: 
\begin{lemma}
\label{pig}
There exists an integer $l \geq K$ such that 
$$\la\left(\te \in  [0,1]: \hau^1_\infty(E_l \cap \Ga_\te) \geq 1/l^2\right) \geq 1/l^2.$$
\end{lemma}
\begin{proof}
Assume for contradiction that no such $l$ exists. Then for each $l$, 
$$\la\left(\te \in  [0,1]: \hau^1_\infty(E_l \cap \Ga_\te\right) \geq 1/l^2)< 1/l^2,$$ and 
$\sum_{l=K}^\infty 1/l^2 <1=\la( [0,1])$, so these sets can not cover $[0,1]$. 
We note that the measurability of the above sets can be checked similarly as in the proof of \cite[Lemma 3.1]{HKM}. 
Therefore, there exists $\te$ such that for all $l \geq K$, $\hau^1_\infty(E_l \cap \Ga_\te)< 1/l^2$. 
But then, using that by construction, $1\leq \hau^1_\infty ([0,1]) \leq  \hau^1_\infty(\Ga_\te) \leq 1$, we have 
$$1=\hau^1_\infty(\Ga_\te) \leq \sum_{l=K}^\infty \hau^1_\infty(E_l \cap \Ga_\te)  < \sum_{l=K}^\infty 1/l^2 <1,$$
which is a contradiction. 
\end{proof}

Fix an $l$ obtained from Lemma \ref{pig}, and let 
$$A_0=\left\{\te \in  [0,1]: \hau^1_\infty(E_l \cap \Ga_\te) \geq 1/l^2\right\},$$ so we have $\la(A_0) \geq 1/l^2$. 
Letting $\ti{\Ga_\te}=\Ga_\te \cap E_l$ for $\te \in A_0$, by construction we have $\hau^1_\infty(\ti{\Ga_\te})\geq 1/l^2$. 
Let $\ti{E}=\bigcup_{\te \in A_0} \ti{\Ga_\te} \su E_l \su E.$

\begin{claim}
\label{neigh}
Let $\de=a^{-l}$, where $l$ is fixed as above. Then $\la(\ti{E}(\de)) \gkb  \frac{\de^{2-d}}{\log^{9}\left(1/\de\right)}.$
\end{claim}
Assuming Claim \ref{neigh} and using that by construction, 
$\ti{E}(\de) \su \bigcup_{i \in J_l} B(x_i,r_i+\de) \su \bigcup_{i \in J_l} B(x_i,2\de)$,  we get the following: 
$$
\frac{\de^{2-d}}{\log^{9}\left(1/\de\right)}  \lkb \la(\ti{E}(\de)) \leq \la\left(\bigcup_{i \in J_l} B(x_i,2\de)\right) \lkb |J_l| \cdot \de^2,
$$
so $|J_l| \gkb \frac{\de^{-d}}{\log^{9} (1/\de)}$. 
Then we have 
\begin{align*}
\sum_{i=1}^{\infty} (2 r_i)^u & \geq   \sum_{i \in J_l} (2 r_i)^u \gkb 
\sum_{i \in J_l} \left(a^{-l}\right)^u = |J_l| \cdot \de^u \gkb \frac{\de^{-d+d-\ep}}{\log^{9}(1/\de)} =\frac{\de^{-\ep}}{\log^{9} (1/\de)} \gkb_\ep 1.
\end{align*}
This means that $\hau^u_\eta(B) \gkb_\ep 1$, and  the proof concludes.

\begin{proof}[Proof of Claim \ref{neigh}] 

Let $\ti{R_\te}=\ti{\Ga_\te}(\de) \su \Ga_\te(\de) \su R_\te$ for $\te \in A_0$. Then $\bigcup_{\te \in A_0} \ti{R_\te}=\ti{E}(\de)$. 
Let $\ti{A}$ be a maximal $\de$-separated subset of $A_0$. Then clearly, 
$A_0 \su \bigcup_{\te \in \ti{A}} (\te-\de,\te +\de).$ 
Therefore we get 
$\la(A_0) \lkb |\ti{A}| \cdot \de$, and by the definition of $A_0$ and $\de$, we obtain that 
\begin{equation}
\label{eqc1}
|\ti{A}| \gkb \la(A_0) \cdot \de^{-1} \geq \frac{1}{l^2} \cdot \de^{-1} \gkb \frac{1}{\log^2(1/\de) \cdot \de}. 
\end{equation}

By the Cauchy-Schwarz inequality, and using that trivially, $\bigcup_{\te \in \ti{A}} \ti{R_\te} \su \ti{E}(\de) $ and $\ti{R_\te}\su R_\te$, we get the following: 
\begin{align}
\label{eqc2}
\sum_{\te \in \ti{A}} \leb(\ti{R_\te}) & \leq \leb\left(\bigcup_{\te \in \ti{A}} \ti{R_\te}\right)^{1/2} \cdot 
\left(\sum_{\te, \te' \in \ti{A}} \leb(\ti{R_{\te'}} \cap \ti{R_\te})\right)^{1/2} \leq \\
& \leq \leb\left(\ti{E}(\de)\right)^{1/2} \cdot \left(\sum_{\te, \te' \in \ti{A}}  \leb(R_{\te'} \cap R_\te)\right)^{1/2}. \nonumber 
\end{align}
By Proposition \ref{areas} applied for the above $\de$ and $A=\ti{A}$, we have 
\begin{equation}
\label{eqck}
\sum\limits_{\te,\te' \in \ti{A}} \leb(R_\te \cap R_{\te'}) \lkb \de^{d-2}\cdot \log(\frac{1}{\de}).
\end{equation}
It remains to prove a lower bound for the left hand side of the Cauchy-Schwarz inequality. 
Let $N(X,\de)$ denote the smallest number of $\de$-disks needed to cover $X \su \rr^2$. Then $\hau^1_\infty(X) \leq N(X,\de) \cdot 2\de$ and 
$\leb(X(\de)) \gkb N(X,\de) \cdot \de^2$. 
Therefore we have 
$$
 \leb(\ti{R_\te})=\leb(\ti{\Ga_\te}(\de)) \gkb  N(\ti{\Ga_\te},\de) \cdot \de^2 \gkb \hau^1_\infty(\ti{\Ga_\te}) \cdot \de  \gkb \frac{\de}{\log^2(1/\de)}.$$
Summing up for $\te \in \ti{A}$ and using \eqref{eqc1}, we get 
\begin{equation}
\label{eqce}
\sum_{\te \in \ti{A}} \leb(\ti{R_\te})  \gkb | \ti{A}| \cdot \frac{\de}{\log^2(1/\de)} \gkb \frac{1}{\log^4(1/\de)}.
\end{equation}
Finally, combining \eqref{eqc2}, \eqref{eqck} and \eqref{eqce} we get 
$$\leb(\ti{E}(\de)) \gkb \frac{1}{\log^8(1/\de)} \cdot \frac{1}{\log(1/\de)} \cdot \de^{2-d}= \frac{\de^{2-d}}{\log^{9}(1/\de)},$$
and this was the statement of Claim \ref{neigh}.

\end{proof}

\textit{Acknowledgment.}
We thank the anonymous referee for the careful reading and valuable comments.

%\clearpage

\end{document}